\documentclass[11pt]{amsart}
\usepackage{geometry}     
\usepackage[parfill]{parskip}    
\usepackage{graphicx}
\usepackage{amssymb}
\usepackage{epstopdf}
\usepackage{xcolor}
\usepackage{fullpage}
\usepackage{tikz-cd}
\usepackage{mathrsfs}
\usepackage[all, 2cell]{xy}
\usepackage{hyperref}
\usepackage{marginnote}

\theoremstyle{plain}
\newtheorem{theorem}{Theorem}[section]
\newtheorem*{theorem*}{Theorem}

\newtheorem{lemma}[theorem]{Lemma}
\newtheorem{corollary}[theorem]{Corollary}
\newtheorem{proposition}[theorem]{Proposition}

\newtheorem{remark}[theorem]{Remark}

\theoremstyle{definition}
\newtheorem{example}[theorem]{Example}

\newcommand{\B}{\mathcal{B}}
\newcommand{\C}{\mathcal{C}}
\newcommand{\D}{\mathrm{D}}
\newcommand{\E}{\mathcal{E}}
\newcommand{\F}{\mathcal{F}}
\newcommand{\G}{\mathcal{G}}

\newcommand{\M}{\mathcal{M}}
\newcommand{\N}{\mathcal{N}}
\renewcommand{\L}{\mathcal{L}}

\newcommand{\R}{\mathcal{R}}
\newcommand{\T}{\mathcal{T}}
\newcommand{\U}{\mathcal{U}}
\newcommand{\V}{\mathcal{V}}

\newcommand{\X}{\mathcal{X}}
\newcommand{\Y}{\mathcal{Y}}
\newcommand{\Z}{\mathcal{Z}}

\renewcommand{\t}{\mathsf{t}}
\newcommand{\f}{\mathsf{f}}

\newcommand{\fp}[1]{{\mathsf{fp}}({#1})}
\renewcommand{\c}{\mathrm{c}}

\newcommand{\Mod}[1]{\mathrm{Mod}#1}

\renewcommand{\mod}[1]{\mathrm{mod}#1}

\newcommand{\Cogen}[1]{\mathrm{Cogen}(#1)}

\newcommand{\Der}[1]{\mathrm{D}(#1)}

\newcommand{\K}[1]{\mathrm{K}^{[0,1]}(\mathrm{Inj}#1)}

\newcommand{\Prod}[1]{\mathrm{Prod}(#1)}
\newcommand{\add}[1]{\mathrm{add}(#1)}
\newcommand{\Kb}[1]{\mathrm{K}^{\mathrm{b}}(\mathrm{Inj}#1)}

\newcommand{\filt}[1]{\mathrm{filt}(#1)}

\newcommand{\y}{\mathrm{y}}

\newcommand{\Hom}[1]{\mathrm{Hom}_{#1}}

\newcommand{\Ext}[1]{\mathrm{Ext}_{#1}}
\newcommand{\End}[1]{\mathrm{End}_{#1}}

\newcommand{\tstr}[1]{\mathbb{T}_{#1}}
\newcommand{\ststr}[1]{\mathbb{D}_{#1}}
\newcommand{\heart}[1]{\mathcal{H}_{#1}}
\renewcommand{\H}[1]{\mathrm{H}^0_{#1}}

\newcommand{\Zg}[1]{\mathbf{Zg}(#1)}
\newcommand{\open}[1]{\mathscr{O}(#1)}

\newcommand{\ZgInt}[1]{\mathbf{Zg}^{[0,1]}(#1)}

\newcommand{\Sp}[1]{\mathbf{Sp}(#1)}
\newcommand{\tors}[1]{\mathbf{tors}(#1)}
\newcommand{\ftors}[1]{\mathbf{f}\mathbf{-tors}(#1)}
\newcommand{\Tors}[1]{\mathbf{f}\mathbf{-Tors}(#1)}
\newcommand{\Cosilt}[1]{\mathbf{2}\mathbf{-Cosilt}(#1)}
\newcommand{\CosiltZg}[1]{\mathbf{MaxRigid}(#1)}

\newcommand{\ClRigid}[1]{\mathbf{ClRigid}(#1)}

\newcommand{\Ann}[1]{\mathrm{Ann}(#1)}
\newcommand{\brick}[1]{\mathbf{brick}(#1)}

\newcommand{\Tpair}[2]{(\mathcal{#1}, \mathcal{#2 })}
\newcommand{\tpair}[2]{(\mathbf{#1}, \mathbf{#2 })}
\newcommand{\Ri}{{\mathbf{Grain}_{\rm cmi}}}
\newcommand{\srigid}{{grain}}

\title{Torsion pairs via cosilting mutation}
\author{Lidia Angeleri H\"ugel}
\address{Dipartimento di Informatica - Settore di Matematica, Universit\`a degli Studi di Verona, Strada le Grazie 15 - Ca' Vignal, I-37134 Verona, Italy} 
\email{lidia.angeleri@univr.it}
\begin{document}
\begin{abstract} For a left artinian ring $A$, we study the lattice $\tors{A}$ of torsion pairs   in the category  of finitely generated $A$-modules  by considering an isomorphic lattice formed by certain closed sets in a topological space associated to $A$, the Ziegler spectrum of the unbounded derived category of $\Mod{A}$. Torsion pairs in $\tors{A}$ turn out to be adjacent if and only if the associated closed sets are related by an operation which is induced by  mutation of cosilting complexes. We describe this operation from several perspectives and present a number of applications in the case when $A$ is a finite dimensional algebra.
\end{abstract}

\def\theequation{\Alph{equation}}
\maketitle
\setcounter{tocdepth}{1}

\section{Introduction}
The torsion pairs in the category $\mod{A}$ of finite dimensional modules over a finite dimensional algebra $A$ form a complete lattice $\tors{A}$ which encodes essential information on $A$. Another important measure for the complexity of the category $\mod{A}$ is given by the set  $\brick{A}$ of isomorphism classes of finite dimensional bricks, i.e.~modules whose endomorphism ring is a skew-field.  

A fundamental  tool for studying $\tors{A}$ and   $\brick{A}$ and their interrelationship  is provided by silting theory. It was shown by Adachi, Iyama and Reiten in \cite{AdachiIyamaReiten:14} that the poset $\ftors{A}$ formed by the  functorially finite torsion pairs is isomorphic to the poset of compact 2-term silting complexes, and torsion pairs are adjacent in $\ftors{A}$ if and only if the associated silting complexes are related by   the operation of silting mutation introduced by Aihara and Iyama in \cite{AiharaIyama}. Moreover,  compact 2-term silting complexes can be represented in the module category by  pairs $(M,P)$  formed by a $\tau$-rigid module $M$ and a projective module $P$.  The $\tau$-rigid modules also provide a link to the collection of bricks:
a result by Demonet, Iyama and Jasso \cite{DemonetIyamaJasso:19} establishes a bijection between  indecomposable $\tau$-rigid modules and   bricks $B$ having the property that the smallest torsion class in $\mod{A}$ containing  $B$ is functorially finite. 

The aim of these notes is to  lift these finiteness conditions and describe the whole lattice $\tors{A}$ and the entire collection $\brick{A}$  in terms of large silting theory. It is more convenient, however, to work with the dual concept of a cosilting complex, since we can then take advantage of the fact that cosilting complexes are pure-injective and work in the Ziegler spectrum, a topological space associated to $A$. 

We will establish a one-one-correspondence between torsion pairs in $\mod{A}$, equivalence classes of 2-term cosilting complexes,  and certain closed  sets    in the Ziegler spectrum of the derived category of  $\Mod{A}$, called \textit{maximal rigid sets}.    
   Torsion pairs are adjacent in $\tors{A}$ if and only if the associated cosilting complexes are related by the operation of (irreducible) cosilting mutation studied in \cite{ALSV}, which extends the silting mutation from \cite{AiharaIyama}.
 Note that, under the correspondence above, functorially finite torsion pairs correspond to maximal rigid  sets formed by finitely many 2-term complexes of finite dimensional injective modules.
    
Let us look at cosilting mutation from a different perspective. A well-known  result of Happel, Reiten and Smal\o\ allows us to attach a t-structure to each torsion pair. The hearts of these t-structures are locally coherent Grothendieck categories with simple objects,  and, roughly speaking, two torsion pairs $\tpair{u}{v}$ and $\tpair{t}{f}$ are adjacent in $\tors{A}$ if and only if the associated hearts $\heart{\mathfrak{u}}$ and $\heart{\mathfrak{t}}$ have a  simple object in common. To be precise, there is a brick $B$ in $\mod{A}$ which gives rise to two simple objects, one in each heart. This is precisely the brick labeling of the Hasse quiver of $\tors{A}$ studied in \cite{DemonetIyamaReadingReitenThomas:23,BarnardCarrollZhu:19}. 

The injective objects in each heart are closely related with the points in the associated maximal rigid set. The points corresponding to injective envelopes of simple objects are called \textit{neg-isolated}.
Cosilting mutation can then be interpreted within the lattice $\CosiltZg{A}$ of maximal rigid sets as an operation that exchanges  neg-isolated points. More precisely,   $\tpair{u}{v}$ and $\tpair{t}{f}$ are adjacent in $\tors{A}$ if and only if the associated maximal rigid sets $\N_{\mathfrak{u}}$ and $\N_{\mathfrak{t}}$ are obtained from each other by exchanging the neg-isolated points that correspond to the injective envelopes of the simple objects arising from $B$ in each heart.

This approach can be extended to a more general situation, where the brick $B$ is replaced by a semibrick and mutation consists in exchanging collections of neg-isolated points. The notion of a \textit{wide interval}  in the lattice $\tors{A}$ introduced by Asai and Pfeifer in \cite{AP} plays a central role in this context. We will connect wide intervals with certain closed rigid sets in the Ziegler spectrum, and we will use the Ziegler topology to determine the neg-isolated points that are mutable. Indeed,  not every torsion pair in $\tors{A}$ has a neighbour, and not every point in a maximal rigid set can be mutated. This is an essential difference between large cosilting theory and $\tau$-tilting theory.

The results outlined above hold true over any left artinian ring $A$. We present three main applications for a finite dimensional algebra. The first establishes a bijection between $\brick{A}$ and certain rigid modules, extending the bijection   from  \cite{DemonetIyamaJasso:19} mentioned above. The second provides new   characterisations of brick-finiteness. The third focuses on the semistable torsion pairs studied by Asai and Iyama in \cite{Asai:21,AsaiIyama}.

The notes are  based on the papers \cite{ALSV,Sentieri:22,ALS1,ALS2} and are organised as follows. In Sections 2 and 3 we collect some preliminaries on torsion pairs and cosilting theory. The correspondence between torsion pairs and maximal rigid sets is explained in Section 4. Section 5 investigates the heart associated to a torsion pair in $\mod{A}$ and lays the foundation for Section 6, where we introduce the neg-isolated points of a maximal rigid set. Section 7 discusses the operation of mutation. The connection between  wide intervals and closed rigid sets is the topic of Section 8. Finally, Section 9 is devoted to applications.

{\bf Notation.}  Throughout the paper, $A$ will be a ring and  all subcategories will be strict and full. We write 
  $ \Mod{A} $  for the category of all left $A-$modules, 
 $ \mod{A} $  for the subcategory of  finitely presented $A-$modules, $ \Der{A} $ for the unbounded derived category of $ \Mod{A}$, and $\Kb{A}$ for the full subcategory of $\Der{A}$ consisting of  complexes of injective modules that are non-zero in only finitely many degrees.

We regard modules as stalk complexes in $ \Der{A} $ concentrated in degree zero. 
Given a class of objects $ \X$ in $\Der{A}$ (or in $\Mod{A}$),  we denote by
  $\Prod{\X} $  the subcategory of $ \Der{A} $ (or of $\Mod{A}$, respectively) formed by the objects isomorphic to direct summands of  products of objects in $ \X $. Moreover, if $I \subseteq \mathbb{Z}$, we write
 $\X^{\perp_I}$ for the subcategory of $ \Der{A} $ (or of $\Mod{A}$, respectively) consisting of the objects $ Y $ with $ \Hom{\Der{A}}(X, Y[i]) = 0 $ (respectively, $\Ext{A}^i (X,Y)=0$) for all $ X \in \mathcal{X} $ and all $i\in I$. We use the notation $>0$ for the interval $\{i\in\mathbb{Z}\mid i>0\}$ and similarly for $\leq 0$.
  The category $ {}^{\perp_I}{\X} $ is defined analogously.  
Finally, for a module $C\in\Mod{A}$ we denote by $\Cogen C$ the subcategory of $\Mod{A}$ consisting of all modules isomorphic to subobjects of products of objects in $C$.

\section{Torsion pairs}\label{tp}

Two subcategories $\T$ and $\F$ closed under direct summands of an abelian or triangulated category $ \mathcal{A} $ form  a \textbf{torsion pair} $\Tpair{T}{F}$ if
$ \Hom{\mathcal{A}}(\T, \F) = 0 $,
and every object $ A \in \mathcal{A} $ is an extension of an object $ T \in \mathcal{T}$ by an object $F \in \mathcal{F} $.
When $ \mathcal{A} $ is abelian, we  call $ \mathcal{T} $ the \textbf{torsion class} and $ \mathcal{F} $ the \textbf{torsionfree class} of the torsion pair. They determine each other as $\mathcal{F}=\mathcal{T}^{\perp_0}$ and 
$\mathcal{T}={}^{\perp_0}\mathcal{F}$. 

 Every class of $A$-modules $\mathcal{M}$  generates a torsion pair $\Tpair{\mathbf  T(\mathcal{M})}{\rm {\mathcal{M}^{\perp_0}}}$ and cogenerates a torsion pair $\Tpair{\rm{}^{\perp_0}\mathcal{M}}{\mathbf F(\mathcal{M})}$ in $\Mod{A}$.
 When $A$ is a left coherent ring  and $\mathcal{M}\subseteq\mod{A}$,   we also have  the torsion class $\tilde{\mathbf T}(\mathcal{M})$  generated by $\mathcal{M}$ in $\mod{A}$, and the torsionfree class
$\tilde{\mathbf F}(\mathcal{M})$    cogenerated by $\mathcal{M}$ in $\mod{A}$.

Torsion classes and torsionfree classes in module categories are determined by closure properties. In particular, over a left noetherian ring $A$, a full subcategory of 
 $\mod{A}$ is a  torsion class if and only if it is closed quotients and extensions. Ordering torsion pairs in $\mod{A}$ by inclusion of torsion classes then gives rise to a complete lattice that we denote by $\tors{A}$. Given two torsion pairs $\tpair{u}{v}$ and $\tpair{t}{f}$ in $\tors{A}$, we set 
 $$\tpair{u}{v}\le \tpair{t}{f}\:\text{ if }\: \mathbf{u}\subseteq \mathbf{t},$$
and  a family $(\tpair{t_i}{f_i})_{i\in I}$ in $\tors{A}$ has meet $(\bigcap_{i\in I} \mathbf{t}_i, \tilde{\mathbf F}(\bigcup_{i\in I} \mathbf{f}_i))$ and  join $(\tilde{\mathbf T}(\bigcup_{i\in I} \mathbf{t}_i), \bigcap_{i\in I} \mathbf{f}_i)$.

 Moreover, when $A$ is left noetherian, every torsion pair $\Tpair{T}{F}$ in $\Mod{A}$ restricts to a torsion pair in $\tors{A}$. Conversely, given a  torsion pair $\tpair{t}{f}$ in $\tors{A}$, we obtain  a torsion pair $\Tpair{T}{F}= (\varinjlim\mathbf t,\varinjlim\mathbf f)$ in  $\Mod{A}$ by taking direct limits. The torsion pairs in $\Mod{A}$ arising in this way are characterised by the property that the torsionfree class $\F$ is closed under direct limits and are called \textbf{torsion pairs of finite type}. They form a complete lattice under the order given by inclusion of torsion classes that here\footnote{Note that this lattice is denoted by $\mathbf{Cosilt}(A)$ in \cite{ALS1,ALS2}.} we will denote by $\Tors{A}$.
  
\begin{theorem}\label{finite type}\cite{Crawley-Boevey:94,BreazZemlicka:18,ZhangWei:17} \cite[Thm.~3.8, Cor.~3.9]{Angeleri:18}\label{CBijection}
 
 (1) If $A$ is a left noetherian ring, there is a lattice isomorphism between $\tors{A}$ and $\Tors{A}$.
 
(2)  A torsion pair $\Tpair{T}{F}$ in $\Mod{A}$ is of finite type if and only if  every left $A$-module  admits an $\F$-cover  (or in other words, a minimal right $\F$-approximation). 
\end{theorem}

\begin{remark}\cite[Cor.~3.8]{BreazZemlicka:18}\label{envel} {\rm
In fact, when $\Tpair{T}{F}$ is a torsion pair of finite type,
every $M\in\Mod{A}$ also admits a short exact sequence 
$\xymatrix{0\ar[r]& M\ar[r]^{f}& N\ar[r]& \bar{N}\ar[r]& 0}$ 
where $f$ is an $\F^{\perp_1}$-envelope (i.~e.~a minimal left  $\F^{\perp_1}$-approximation), and  $\bar{N}\in\F$.
} 
\end{remark}

Torsion pairs in $\mod{A}$, even if not functorially finite,  thus give rise to approximations when we pass to the category of all modules $\Mod{A}$. This suggests that  silting theory can still provide a tool to study the lattice $\tors{A}$ when dropping   finiteness assumptions.  

\section{Cosilting complexes}\label{cosilt}
Recall that a torsion pair $\tstr{}=\Tpair{X}{Y}$ in $\Der{A}$ is a \textbf{t-structure} if $\X$ is closed under positive shifts. We  call $ \mathcal{X} $ the \textbf{aisle} and $ \mathcal{Y} $ the \textbf{coaisle} of the t-structure, and  $\heart{\tstr{}} := \Y \cap \X[-1]$ is its \textbf{heart}. It is an abelian category with a cohomological functor $\H{\tstr{}} \colon \Der{A} \to \heart{\tstr{}}$. 

The \textbf{standard t-structure}  in $\Der{A}$ will be denoted by $\ststr{} := (\D^{<0}, \D^{\geq 0})$, where 
$\D^{<0} = \{X \in \Der{A} \mid \mathrm{H}^i(X) = 0 \text{ for all } i\geq 0\}$ and  $\D^{\geq0} = \{X \in \Der{A} \mid \mathrm{H}^i(X) = 0 \text{ for all } i<0\}$.

We will also need the following construction due to Happel, Reiten and Smal{\o}.
We start with a   t-structure $\tstr{} = (\X, \Y)$  in $\Der{A}$ which is non-degenerate, i.e.~$\bigcap_{i\in\mathbb{Z}}\X[i] = \{0\} = \bigcap_{i\in\mathbb{Z}}\Y[i]$, together with
 a torsion pair  $\mathfrak t = (\T, \F)$ in the heart $\heart{\tstr{}}$.  The \textbf{(right) HRS-tilt} of $\tstr{}$ at $\mathfrak t$ is the t-structure $\tstr{\mathfrak t^-} = (\X_{\mathfrak t^-}, \Y_{\mathfrak t^-})$ given by
	\[\X_{\mathfrak t^-} = \{ X \in \Der{A} \mid \H{{\tstr{}}}(X) \in \T \text{ and } \mathrm{H}^k_{\tstr{}}(X) = 0 \text{ for all } k>0\},\]
	\[\Y_{\mathfrak t^-} = \{ X \in \Der{A} \mid \H{{\tstr{}}}(X) \in \F \text{ and } \mathrm{H}^k_{\tstr{}}(X) = 0 \text{ for all } k<0\}.\]

We are now ready to introduce cosilting complexes.  

\begin{proposition}\cite{ZhangWei:17,PsaroudakisVitoria:18,NSZ,MarksVitoria:18}
The following statements are equivalent for   $\sigma\in\Kb{A}$.
\begin{enumerate}
\item[(i)] The pair $\mathbb T_\sigma=({}^{\perp_{\leq 0}}\sigma, {}^{\perp_{> 0}}\sigma)$ is a t-structure in $\Der{A}$.
\item[(ii)]  $\Hom{\Der{A}}(\sigma^I, \sigma[1]) = 0$ for all sets $I$, and $\Kb{A}$ agrees with the smallest thick (i.e.~closed under  direct summands, extensions, and shifts)
 subcategory of $\Der{A}$ that contains $\Prod\sigma$\end{enumerate}
\end{proposition}
A complex  $\sigma$
 with these properties is called a \textbf{cosilting complex}. We say that two cosilting complexes $\sigma,\sigma'$ are \textbf{equivalent} if they give rise to the same t-structure. This amounts to the condition $\Prod{\sigma}=\Prod{\sigma'}$.

We are  going to see that torsion pairs of finite type correspond bijectively to 2-term cosilting complexes, up to equivalence. Let us denote by $\K{A}$  the subcategory of $\Kb{A}$ consisting of objects isomorphic to complexes that are non-zero in degrees $0$ and $1$ only, and by $\Cosilt{A}$  the collection of equivalence classes of cosilting complexes in $\K{A}$. 
Note that every complex $\sigma\in\K{A}$, when viewed as a morphism between injective modules,  induces a subcategory  of $\Mod{A}$ $$\C_{\sigma}:=\{M\in \Mod{A}\mid \Hom{A}(M, \sigma) \text{ is surjective}\}$$  which is closed under subobjects and extensions and can be regarded as a ``shadow'' of  ${}^{\perp_{> 0}}\sigma$.
  
\begin{lemma}\label{csigma} A complex $\mu \in \K{A}$ satisfies $\Hom{\Der A}(\mu, \sigma[1]) = 0$
   if and only $\H{}(\mu)\in\C_{\sigma}$. \end{lemma}

Let now
$\Tpair{T}{F}$ be a  torsion pair of finite type in $\Mod{A}$.
We fix an injective cogenerator $E$ of $\Mod{A}$ and take an $\F$-cover $g:C_0 {\longrightarrow} E$ of $E$. Consider   the exact sequence \begin{equation}0\longrightarrow C_1 \overset{}{\longrightarrow} C_0 \overset{g}{\longrightarrow} E\end{equation} and observe that  $C_1\in \F\cap\F^{\perp_1}$ by Wakamatsu's Lemma.  We lift this sequence to $\Der{A}$ by taking  a minimal injective copresentation $\mu_0$   of $C_0$  together with the triangle  \begin{equation}\label{triangle}\mu_1 \longrightarrow \mu_0 \overset{g'}{\longrightarrow} E \longrightarrow \mu_1[1]\end{equation}   completed from the morphism $g' \colon \mu_0 \longrightarrow E$ induced by  $g$. We claim  that the complex  $$\sigma:=\mu_1\oplus\mu_0\in \K{A}$$  is a cosilting complex whose zero cohomology $C:= {\H{}(\sigma)}=C_1\oplus C_0$ cogenerates  $\F$.
Indeed, one can check as in the proof of \cite[Prop.~3.5]{Angeleri:18} that $\C_\sigma=\C_{\mu_1}=\Cogen{C_0}$, and from the approximation property of $g$ it follows that $\F=\Cogen{C_0}=\Cogen C$.
Thus $\Tpair{T}{F}=\Tpair{\rm {{}^{\perp_0}C}}{{\C_\sigma}}$, and   arguments dual to \cite[Thm.~4.9]{AngeleriMarksVitoria:16} show that the HRS-tilt of the standard t-structure at the torsion pair  $\Tpair{T}{F}$ coincides with the t-structure $\mathbb T_\sigma=({}^{\perp_{\leq 0}}\sigma, {}^{\perp_{> 0}}\sigma)$. This proves that $\sigma$ is a cosilting complex with the stated property. 

The modules $C$ arising as zero cohomologies of a 2-term cosilting complex $\sigma$  are called \textbf{cosilting modules}. They are the large counterpart of the support $\tau^{-}$-tilting modules from \cite{AdachiIyamaReiten:14}.  Building on a fundamental result for cotilting modules in \cite{Bazzoni}, it is shown in \cite{BreazPop:17,ZhangWei:17} that cosilting modules are always pure-injective. This implies that the categories of the form $\C_\sigma=\Cogen C$  are always closed under direct limits, see \cite[Cor.4.8]{BreazPop:17}.

Now we can consider the assignment which takes  a 2-term cosilting complex $\sigma$ to the torsion pair of finite type $\Tpair{\rm {{}^{\perp_0}C}}{{\C_\sigma}}$ cogenerated by $C=\H{}(\sigma)$.  Since $\H{}$ commutes with direct products, we have that equivalent cosilting complexes are mapped to the same torsion pair, and by the discussion above we obtain the following result.

\begin{theorem} \cite[Thm.~2.17]{ALS1}\label{cosiltbij} {There is a bijection between $\Cosilt{A}$ and  $\Tors{A}$.} If $\mathfrak{t}=\Tpair{T}{F}$ is the torsion pair associated to a cosilting complex $\sigma$, then  the t-structure $\mathbb T_\sigma=({}^{\perp_{\leq 0}}\sigma, {}^{\perp_{> 0}}\sigma)$  coincides with $\ststr{\mathfrak{t}^-}$, the HRS-tilt of the standard t-structure $\ststr{}$ at the torsion pair  $\mathfrak t$.
 \end{theorem}

\section{Maximal rigid sets}\label{zg}

We have already mentioned above that zero cohomologies of 2-term cosilting complexes are pure-injective. In fact, all cosilting complexes are pure-injective by \cite{MarksVitoria:18}, and this allows us to use the theory of purity in order to study the lattice $\tors{A}$. Let us recall some notions from \cite{Krause:00}.

A \textbf{pure triangle} is a triangle $ X\stackrel{f}{\to} Y\stackrel{g}{\to} Z\to X[1]$ in $\Der{A}$ such that $0\to \Hom{\Der{A}}(F, X)\to \Hom{\Der{A}}(F,Y)\to \Hom{\Der{A}}(F,Z)\to 0$ is an exact sequence of abelian groups  for every compact object $F\in\Der{A}$. 
In this case $f$ is called a \textbf{pure monomorphism} and $g$ a \textbf{pure epimorphism}. 
An object $X$ in 
$\Der{A}$ is \textbf{pure-injective} if every pure triangle  starting at $X$  splits. 
A subcategory $\C$ of $\Der{A}$ is \textbf{definable} if it is closed under direct products, pure monomorphisms and pure epimorphisms.

The \textbf{Ziegler spectrum} $\Zg{\Der{A}}$ of $\Der{A}$ 
is the topological space whose points are given by the isomorphism classes of indecomposable pure-injective objects in $\Der{A}$ and whose closed sets  are the sets
of the form $\C \cap\Zg{\Der{A}}$  where $\C$ is a definable subcategory of $\Der{A}$.
The subset formed by the isomorphism classes of indecomposable pure-injective complexes concentrated in degrees 0 and 1 will be denoted by
   $\ZgInt{A} := \K{A}\cap\Zg{\Der{A}}$.
   
These notions are triangulated versions of the more classical notions of purity in module categories. In particular, the Ziegler spectrum $\Zg{A}$ of $\Mod{A}$ is defined by taking as points the isomorphism classes of indecomposable pure-injective modules, and the closed sets are defined correspondingly.

We now collect some properties of cosilting complexes. We will say that a set $\M$ of indecomposable objects in $\Der{A}$ is \textbf{rigid} if $\Hom{\Der{A}}(\mu, \nu[1]) = 0$ for all $\mu, \nu \in \M$.

\begin{proposition}\label{propcosilt} The following statements hold true for  a cosilting complex $\sigma\in\K{A}$ with associated t-structure $\mathbb T_\sigma=({}^{\perp_{\leq 0}}\sigma, {}^{\perp_{> 0}}\sigma)$. 
\begin{enumerate}
\item 
\cite{MarksVitoria:18} The complex $\sigma$ is pure-injective in $\Der{A}$.

\item\cite[Lemma 4.8]{AngeleriMarksVitoria:17} The coaisle $\Y$  is a definable subcategory of $\Der{A}$ with $\Y\cap\Y^{\perp_1}=\Prod{\sigma}$.

\item \cite[Prop.~3.2]{ALS1} The set $\N:=\Prod{\sigma}\cap\Zg{\Der{A}}$   is  a closed set  in $\Zg{\Der{A}}$ such that $\Prod\N=\Prod\sigma$, and the t-structure $\mathbb T_\sigma$  coincides with  $({}^{\perp_{\leq 0}} \N, {}^{\perp_{> 0}} \N)$.

\item \cite[Thm.~3.8]{ALS1} The set $\N$ is  rigid, and it is maximal among all rigid sets in $\ZgInt{A}$.
\end{enumerate}
\end{proposition}

Let us denote by $\CosiltZg{A}$ the collection of all maximal rigid sets in $\ZgInt{A}$ with the partial order  $\M \leq\N$  when $\Hom{\Der{A}}(\N, \M[1])=0$.

\begin{theorem}\cite[Thm.~3.8]{ALS1}\label{bijcpxmaxrigid} 
If $A$ is a left artinian ring, the assignment that  takes  a cosilting complex $\sigma$ to the maximal rigid set $\N:=\Prod{\sigma}\cap\Zg{\Der{A}}$ defines a bijection between $\Cosilt{A}$ and $\CosiltZg{A}$.
\end{theorem}
 The reverse assignment associates to $\N\in\CosiltZg{A}$ the complex $\sigma := \prod_{\omega \in \mathcal{N}}\omega$ in $\K{A}$. We need the assumption that $A$ is left artinian to ensure that the rigid object  $\sigma$ is indeed a cosilting complex. In fact, the artinian property allows us to argue that the class $\C_\sigma\cap\mod{A}$, being closed under submodules and extensions,  is a torsionfree class in $\mod{A}$. Now we infer as in \cite[Lemma 4.2]{AngeleriSentieri:23+} that $\C_\sigma$ is a torsionfree class in $\Mod{A}$, and we use Bongartz completion (cf.~\cite[Prop.~3.10]{ZhangWei:17}) to see that $\sigma$ is a direct summand of a cosilting complex $\mu\in\K{A}$. The maximality of  $\N$ then implies that the maximal rigid set associated to $\mu$ coincides with $\N$, and $\sigma$ is a cosilting complex equivalent to $\mu$. 
 
 Combining Theorems~\ref{CBijection}, \ref{cosiltbij}, and \ref{bijcpxmaxrigid}, we obtain the following result.
 
 \begin{corollary}\cite[Cor.~3.11]{ALS1}\label{fundbij} If $A$ is a left artinian ring, there is a lattice isomorphism between $\tors{A}$ and $\CosiltZg{A}$. \end{corollary}
 \begin{proof} In order to show that the composition of the bijections established above is order preserving, we consider two torsion pairs $\tpair{u}{v}\le\tpair{t}{f}$ in $\tors{A}$. They are mapped via Theorem~\ref{CBijection} to torsion pairs  $\mathfrak{u}:=\Tpair{U}{V}$ and $\mathfrak{t}:=\tpair{T}{F}$ in $\Tors{A}$, where clearly $\mathfrak u\le \mathfrak t$. Applying Theorems~\ref{cosiltbij}, and \ref{bijcpxmaxrigid}, we obtain two cosilting complexes $\sigma_{\mathfrak{u}}$ and  $\sigma_{\mathfrak{t}}$ with associated maximal rigid sets $\N_{\mathfrak{u}}$ and $\N_{\mathfrak{t}}$. Recall that the t-structures $\mathbb T_{\sigma_{\mathfrak{u}}}$ and $\mathbb T_{\sigma_{\mathfrak{t}}}$ coincide on one hand with the t-structures defined by $\N_{\mathfrak{u}}$ and $\N_{\mathfrak{t}}$, and on the other hand with the HRS-tilts of the standard t-structure at $\mathfrak u$ and $\mathfrak t$, respectively. So  $\mathfrak u\le \mathfrak t$ amounts to an inclusion of aisles, or equivalently, a reverse inclusion of  coaisles ${}^{\perp_{> 0}} \N_{\mathfrak{u}}\supseteq{}^{\perp_{> 0}} \N_{\mathfrak{t}}$. But this just means that  ${}^{\perp_{> 0}} \N_{\mathfrak{u}}$ contains $\N_{\mathfrak{t}}$, because  ${}^{\perp_{> 0}} \N_{\mathfrak{t}}$ is the smallest coaisle in $\Der{A}$ containing  $\N_{\mathfrak{t}}$  by \cite[Prop.~4.9]{PsaroudakisVitoria:18}. Now we have $ \N_{\mathfrak{u}}\le  \N_{\mathfrak{t}}$ by  definition.\end{proof}

 \begin{remark}\label{indecomp}
{\rm  (1) \cite[Lem.~4.4]{ALS1} Every complex in $\ZgInt{A}$   is isomorphic  to the  minimal injective copresentation  $\mu_M $ of a module $M\in\Zg{A}$, or to $I[-1]$ for an  indecomposable injective module $I$. 

Corollary~\ref{fundbij} can then be rephrased within the module category $\Mod{A}$, obtaining  a bijection between torsion pairs in $\mod{A}$ and certain pairs $(\Z,\mathcal I)$, called \textbf{cosilting pairs},  where $\Z$ is  a closed set  in the Ziegler spectrum  $\Zg{A}$  and $\mathcal I$ is a set of indecomposable injective modules, see \cite[\S 4]{ALS1}.}

\label{ff}{\rm (2)   Let $A$ be a finite dimensional algebra over a field. Recall from \cite {AdachiIyamaReiten:14} that the poset $\ftors{A}$ is parametrised  by  support $\tau$-tilting pairs, or dually, by  support $\tau^{-}$-tilting pairs. The bijection  between torsion pairs and cosilting pairs considered in (1) is  a natural generalisation of this parametrisation. Indeed, the functorially finite torsion pairs correspond  to the cosilting pairs  $(\Z,\mathcal I)$ where   $\Z$ is a finite set and consists of finitely generated  modules, see \cite[Prop.~4.11]{ALS1}. 
In other terms, the
 bijection in Corollary~\ref{fundbij} restricts to a bijection between functorially finite torsion pairs and maximal rigid sets formed by finitely many two-term complexes of finite dimensional injective modules.  }
\end{remark}

\section{Hearts}\label{hearts}

 Now that we have a one-one-correspondence between torsion pairs in $\mod{A}$ and maximal rigid sets, we want to understand how properties of the lattice $\tors{A}$ are reflected in $\CosiltZg{A}$. We will see in Section~\ref{negis} that every maximal rigid set $\N$ has some distinguished elements, called neg-isolated points, that  play an important role when investigating the shape of $\tors{A}$. This will become more clear once we uncover the information encoded by   the  associated t-structure
$({}^{\perp_{\leq 0}}\N, {}^{\perp_{> 0}}\N)$. Recall that this is the t-structure induced by  the corresponding cosilting complex.

It will be convenient to investigate this t-structure from the perspective of localisation theory in Grothendieck categories. To this end, we consider the functor category $\Mod{(\D^\c)}$ of   $\D^\c$. Here $\D^\c$ denotes the subcategory of $\Der{A}$ formed by the compact objects,
and  $\Mod{(\D^\c)}$ is the category of additive contravariant functors from $\Mod{(\D^\c)}$ to the category of abelian groups. It is well known that $\Mod{(\D^\c)}$ is a locally coherent Grothendieck category.

When we have a locally coherent Grothendieck category $\G$, we can 
consider  the set $\Sp{\G}$ formed by the isomorphism classes of indecomposable injective objects in $\G$. It becomes a topological space, called the spectrum of $\G$,  
by taking as a basis of open sets 
the collection of sets $$\open{C} = \{E \in \Sp{\G} \mid \Hom{\G}(C, E) \neq 0 \}$$ where $C$ runs through the finitely presented objects in $\G$. There is  a bijection between
the closed sets of $\Sp{\G}$ and the hereditary torsion pairs (i.e.~torsion pairs whose torsion class is closed under subobjects)   of finite type in $\G$. This bijection takes a closed set $\E$ to the torsion pair $({}^{\perp_0}\E,\Cogen\E)$ in $\G$ cogenerated by $\E$. Then ${}^{\perp_0}\E$ is a localising subcategory and one can construct the quotient  $\G/{}^{\perp_0}\E$, which is again a locally coherent Grothendieck category. We refer to \cite{Herzog:97,Krause:97} for details. 

 Let us return to our t-structure. \textit{From now on $A$ denotes an artinian ring.} We fix  a  torsion pair  $\tpair{t}{f}$ in $\tors{A}$, and we use the following notation.
 \begin{itemize}  
\item $\mathfrak t=\Tpair{T}{F}$ is the associated torsion pair  in $\Tors{A}$,
\item  $\sigma_{\mathfrak t}$ is the associated  cosilting complex,
\item $\N_{\mathfrak t}$ is the associated maximal rigid set,
\item  $\heart{\mathfrak t}$ is  the  heart  of the t-structure $\mathbb T_{\mathfrak t}:=({}^{\perp_{\leq 0}}\sigma_{\mathfrak t}, {}^{\perp_{> 0}}\sigma_{\mathfrak t})$,
\item $\H{\mathfrak t}:\Der{A}\longrightarrow\heart{\mathfrak t}$ is the cohomological functor.
\end{itemize}

The Yoneda functor $\y:\Der{A}\longrightarrow \Mod{(\D^\c)}$  maps the pure-injective object $\sigma_{\mathfrak t}$ to an injective object $E:=\y\sigma_{\mathfrak t}$ in the functor category $\Mod{(\D^\c)}$ and induces a homeomorphism
$$ \y \colon \Zg{\Der{A}} \to \Sp{\Mod{(\D^\c)}}$$ which takes the closed set $\N_{\mathfrak t}$ to the set $\Prod E\cap\Sp{\Mod{(\D^\c)}}$. This means that $E$ cogenerates a hereditary torsion pair of finite type $({}^{\perp_0}E,\Cogen{E})$  in $\Mod{(\D^\c)}$, and we can consider the  localisation  $q:\Mod{(\D^\c)}\longrightarrow\Mod{(\D^\c)}/{}^{\perp_0}E$.
It is shown in \cite{AngeleriMarksVitoria:17} that 
$\Mod{(\D^\c)}/{}^{\perp_0}E$ is equivalent to the heart $\heart{\mathfrak t}$ of $\mathbb T_{\mathfrak t}$ via the functor $F$ in the following commutative diagram:
	\[\xymatrix{ \Der{A} \ar[r]^-\y \ar[d]_{\H{\mathfrak t}} & \Mod{(\D^\c)} \ar[r]^-q & \Mod{(\D^\c)}/{}^{\perp_0}E \ar[dll]^-F \\ \heart{\mathfrak t} & & }\]

 We thus obtain the following result.

\begin{proposition}  \cite[Prop.~3.2]{ALS1}\label{propheart} The   heart $\heart{\mathfrak t}$ of $\mathbb T_{\mathfrak t}$ is a locally coherent Grothendieck category, and the cohomological functor  $\H{\mathfrak t}:\Der{A}\longrightarrow\heart{\mathfrak t}$ induces a homeomorphism between $\N_{\mathfrak t}$ and  $\Sp{\heart{\mathfrak t}}$.
\end{proposition}

Recall that $\mathbb T_{\mathfrak t}$ is the HRS-tilt of the standard t-structure at the torsion pair $\mathfrak{t}$. So   there is a tilted  torsion pair $(\F, \T[-1])$  in $\heart{\mathfrak t}$.
Moreover, by \cite{Saorin:2017}  the class of finitely presented objects in $\heart{\mathfrak t}$ is $$\fp{\heart{\mathfrak t}}=\heart{\mathfrak t}\cap D^b(\mod{A}).$$
In particular, if  $\mathbf f\not=0$,
the modules in   $\mathbf f$ become finitely presented torsion objects in $\heart{\mathfrak t}$. But every finitely generated object in $\heart{\mathfrak t}$ has a maximal subobject, hence a simple factor. So we conclude that $\heart{\mathfrak t}$ contains simple torsion objects.

Let us describe the simple objects in $\heart{\mathfrak t}$.
 We say that a non-zero module $ B $ is \textbf{torsionfree almost torsion} with respect to the torsion pair $\mathfrak t$ if the following conditions are satisfied:
\begin{enumerate}
\item[(i)] $ B \in \mathcal{F} $, but every proper quotient of $ B $ is contained in $ \mathcal{ T } $; 
\item[(ii)] in every short exact sequence $ 0 \to B \to F \to M \to 0 $ with $ F \in \mathcal{F} $ we have $ M \in \mathcal{F} $.
\end{enumerate} 
Dually, one defines \textbf{torsion, almost torsionfree} modules.
In both cases $B$ is a {brick}.

\begin{proposition}\label{simples in the heart}  
\cite[Theorem 3.6]{AHL}  The simple objects in  $\heart{\mathfrak t}$  are precisely\\
- the objects of the form $F$ where $F\in\Mod{A}$ is torsionfree almost torsion with respect to $\mathfrak{t}$, and\\  
- the objects of the form $T[-1]$   where $T\in\Mod{A}$ is torsion almost torsionfree with respect to $\mathfrak{t}$.
\end{proposition}

Observe that   torsionfree almost torsion modules always exist by the argument above, but they need not be finitely generated,  while torsion almost torsionfree modules are always finitely generated, but need not exist.
Moreover, in the category $\mod{A}$ over a finite dimensional algebra $A$, these modules coincide with  the  minimal extending and minimal coextending modules  considered in \cite{BarnardCarrollZhu:19}, and   they label minimal inclusions of torsion pairs. We will come back to this later.

\begin{example}\label{kron}
 Let  $ A $ be the Kronecker algebra, i.e.~the path algebra of the quiver $Q:\xymatrix{\bullet\ar@<-.7ex>[r]\ar@<.7ex>[r]&\bullet}$ over an algebraically closed field $k$. Let $\mathbf{p}$, $\mathbf{r}$, and $\mathbf q$  denote the classes of indecomposable  preprojective,  regular, and  preinjective modules, respectively.

Consider the torsion pair $\tpair{t}{f}:=\tpair{\add{\mathbf{r}\cup\mathbf{q}}}{\add{\mathbf p}}$ in $\mod{A}$, and let 
$\mathfrak{t}=\Tpair{T}{F}$ be the associated torsion pair in $\Tors{A}$. The torsion almost torsionfree modules with respect to $\mathfrak{t}$ are the simple regular modules, and the only  torsionfree almost torsion module is the (infinite dimensional) generic module $G$. 

Now consider the torsion pair $\tpair{u}{v}:=\tpair{\add{\mathbf{q}}}{\add{\mathbf p\cup\mathbf{r}}}$ in $\mod{A}$, and let 
$\mathfrak{u}=\Tpair{U}{V}$ be the associated torsion pair in $\Tors{A}$. The  torsionfree almost torsion modules  with respect to $\mathfrak{u}$ are the simple regular modules, and there are no torsion almost torsionfree modules. 
\end{example}

\section{Neg-isolated points}\label{negis}
 Our next aim is to determine the injective envelopes of simple objects in the heart $\heart{\mathfrak t}$. Notice that the injective envelope of a simple object $B$  gives rise to an object in $\Sp{\heart{\mathfrak t}}$ which, according to Proposition~\ref{propheart}, is the image of a point in the maximal rigid set $\N_{\mathfrak t}$ under the homeomorphism induced by $\H{\mathfrak t}$. We will call this point of $\N_{\mathfrak t}$  the \textbf{$\N_{\mathfrak t}$-injective envelope} of $B$.

We fix again  a torsion pair  $\tpair{t}{f}$ in $\tors{A}$ and adopt the  notation from \S~\ref{hearts}.

Recall that a morphism $ f : N \to \overline{N} $ in $ \mathcal{F} $ is \textbf{left almost split in $ \mathcal{F} $}  if it is not a split monomorphism and  every morphism $ g : N \to N' $ in $\F$ which is not a split monomorphism factors through $f$.
 We say that $N\in\F$ is \textbf{neg-isolated} in $\F$ if there exists a left almost split map $N\to\overline{N}$ in $\F$.
Moreover, 
 $N\in\F$ is called \textbf{critical} in $\F$ if there exists a left almost split map $N\to\overline{N}$ in $\F$ that is an epimorphism. This implies that $N\in\F \cap \F^{\perp_1}$.
Finally, a module $N\in\F\cap\F^{\perp_1}$ is called \textbf{special} in $\F$ if there exists a left almost split map  $N\to\overline{N}$ in $\F$ that is a monomorphism.
  It follows from \cite[Lem.~4.3]{AHL} that  every module $N\in \F\cap\F^{\perp_1}$ that is neg-isolated in $\F$ is either critical or special. 

There is a nice interplay between simple objects in  $\heart{\mathfrak t}$ and neg-isolated modules in $\F$. Recall from \S~\ref{tp} that every module admits an $\F$-cover and an $\F^{\perp_1}$-envelope. The latter is always injective, whereas the former need not be surjective. In fact, $M$ has a surjective $\F$-cover if and only if $M$ is a module over $A/\Ann{\F}$, where $\Ann{\F}$ denotes the annihilator of $\F$.  
 \begin{proposition}\label{las} \cite[Prop.~5.10]{ALS1}\cite[Prop.~4.6]{ALS2}
 Let the notation be as in \S~\ref{hearts}. 
 \begin{enumerate}
\item Given a short exact sequence
$
\begin{tikzcd}
0 \arrow[r] & F \arrow[r, "f"] & N \arrow[r, "a"] & \overline{N} \arrow[r] & 0
\end{tikzcd}
$ in $\Mod{A}$, we have that $ F $ is torsionfree, almost torsion with respect to $\mathfrak{t}$ and $ f $ is its $ \F^{\perp_1}-$envelope if and only if $N$ is critical in $\F$ and $a$ is a left almost split map in $\F$. 
\item  Given a short exact sequence  
$
\begin{tikzcd}
0 \arrow[r] & N \arrow[r, "b"] & \overline{N} \arrow[r, "g"] & T \arrow[r] & 0
\end{tikzcd}
$ in $\Mod{A}$, we have that $ T $ is torsion, almost torsionfree with respect to $\mathfrak{t}$ and $ g $ is its $ \F-$cover if and only if $N$ is special in $\F$ and $b$ is a  left almost split map in $\F$. 
\item If $ T $ is torsion, almost torsionfree with respect to $\mathfrak{t}$  and its $ \mathcal{F}-$cover 
is not surjective, 
then $ T $ has a unique maximal submodule $L$, and the embedding $g:L\hookrightarrow T$ 
 is an $\F$-cover.
 
\end{enumerate}
\end{proposition}

We can now compute the   $\N_{\mathfrak t}$-injective envelopes. 
If $\mu\in\N_{\mathfrak t}$ is the $\N_{\mathfrak t}$-injective envelope of 
a simple object $B$ in $\heart{\mathfrak t}$, then $H^0_{\mathfrak t}(\mu)$
 is the unique element  of $\Sp{\heart{\mathfrak t}}$ with $\Hom{\heart{\mathfrak t}}(B, H^0_{\mathfrak t}(\mu))\not=0$.  
 Since $\Hom{\Der{A}}(B, \mu)\cong\Hom{\heart{\mathfrak t}}(B, H^0_{\mathfrak t}(\mu))$, we are thus looking for the unique element $\mu\in\N_{\mathfrak t}$ with $\Hom{\Der{A}}(B, \mu)\not=0$. 
  
Suppose that $B=F$ as in case (1) above. It is easy to see that the   minimal injective copresentation $\mu_N$ of the critical module $N$ associated to $F$   lies in $\N_{\mathfrak t}$ and is unique with $\Hom{\Der{A}}(F,\mu_N)\cong \Hom{A}(F,N)\not=0$.  Similarly, in case (2) the special module  $N$ gives rise to the unique element  $\mu_N\in\N_{\mathfrak t}$ such that $\Hom{\Der{A}}(T[-1],\mu_N)\not=0$.
Finally, in case (3),   the injective envelope $I$ of the simple $ A-$module $T/L$ 
induces the unique element $I[-1]\in\N_\mathfrak{t}$ such that $\Hom{\Der{A}}(T[-1],I[-1])\cong \Hom{A}(T, I) \ne 0 $. An indecomposable injective module that arises from a torsion, almost torsionfree module as in case (3) above  will be called \textbf{special}. For a characterisation of such injective modules we refer to \cite[\S 4]{ALS2}. Over a finite dimensional algebra they are precisely the indecomposable injective modules in $\F^{\perp_0}$.

\begin{proposition}\cite[Cor.~4.8]{ALS2} \mbox{The  $\N_{\mathfrak t}$-injective envelopes of simples in $\heart{\mathfrak t}$ are given as follows.} 
$$\begin{array}{c|c|c}
\text{brick} &\textrm{simple object in}\;\heart{\mathfrak t} & \textrm{$\N_{\mathfrak t}$-injective envelope}  \\
\hline\hline
\text{$F$ torsionfree almost torsion}  &F& \mu_N\; \text{with $N$  critical}\\
\hline
\text{$T$ torsion almost torsionfree, surj.~$\F$-cover}  & T[-1] &\mu_N\;\text{with $N$ special}\\
\hline
\text{$T$ torsion almost torsionfree, inj.~$\F$-cover} &T[-1]& I[-1]\;\text{with $I$  special }
 \end{array}$$
  
\end{proposition}

\medskip

\begin{remark} \cite[Thm.~4.7(2)]{ALS2} {\rm 
We have seen in Remark~\ref{indecomp} that there   a cosilting pair $(\Z,\mathcal I)$ associated with $\tpair{t}{f}$. 
When $A$ is an artin algebra, one can use results from \cite{DemonetIyamaReadingReitenThomas:23,BarnardCarrollZhu:19} to show that the special injectives are precisely the modules  in $\mathcal I$. }
\end{remark}

We will say that a point  $\mu\in \N_{\mathfrak t}$ is \textbf{critical} in $\N_{\mathfrak t}$
if $\mu =\mu_N$ for a critical module in $\F$, and $\mu$ is \textbf{special} in $\N_{\mathfrak t}$
if $\mu =\mu_N$ for a special module in $\F$, or $\mu=I[-1]$ for a special injective module $I$. The points of $\N_{\mathfrak t}$ that are special or critical in $\N_{\mathfrak t}$ will be called  \textbf{neg-isolated} in $\N_{\mathfrak t}$. They are precisely the $\N_{\mathfrak t}$-injective envelopes of simples, and they can be detected as follows. \begin{proposition}\cite[Prop.~4.14]{ALS2} Let $E$ be an injective cogenerator of $\Mod{A}$, and let $$\mu_1 \to \mu_0 \to E \to \mu_1[1]$$ be the approximation triangle~(\ref{triangle}) constructed as in Section~\ref{cosilt} from an $\F$-cover of $E$.  
The critical elements of $\N_{\mathfrak t}$ are precisely the direct summands of $\mu_0$ that are neg-isolated in $\N_{\mathfrak t}$, and the  special elements of $\N_{\mathfrak t}$ are precisely the direct summands of $\mu_1$ that are neg-isolated in $\N_{\mathfrak t}$. \end{proposition}

\section{Mutation}\label{muta}

We are now ready to discuss how the lattice structure of $\tors{A}$ is reflected in $\CosiltZg{A}$ via the isomorphism in Corollary~\ref{fundbij}. This will expressed by an operation on $\CosiltZg{A}$ that corresponds to the mutation of cosilting complexes studied in \cite{ALSV}, which in turn is inspired by the operation of silting mutation introduced in \cite{AiharaIyama}.

We say that two maximal rigid sets $\L$ and $\R$ in $\Der{A}$ are \textbf{related by mutation} (and $\R$ is a right mutation of $\L$, while $\L$ is a left mutation of $\R$) if there is a bijection $\tt m \colon \L \to \R$ such that $\tt m\mid_{\L\cap\R}$ is the identity map, and  $\lambda\in \L\setminus \R$ is taken  to an element $\rho:=\tt m(\lambda)$ given by the \mbox{exchange triangle} 
	\begin{equation}\label{eq: exchange triangle}\rho \overset{f}{\longrightarrow} \epsilon_\lambda \overset{g}{\longrightarrow} \lambda \to \Delta(\lambda)[1]\end{equation}
where $g \colon \epsilon_\lambda \to \lambda$ is a minimal right $\Prod{\L\cap\R}$-approximation of $\lambda$ and $f \colon \rho \to \epsilon_\lambda$ is a minimal left $\Prod{\L\cap\R}$-approximation of $\rho$. Equivalently, the  cosilting complexes $\sigma_{\L}$ and  $\sigma_{\R}$ associated to $\L$ and  $\R$, respectively,   are related by  cosilting mutation according to the definition in \cite{ALSV}, as shown in  \cite[Lemma 3.2]{ALS2}. 
We say that the mutation is \textbf{irreducible} if $|\L\setminus\R|= 1$.

Let us now  fix two torsion pairs $\tpair{u}{v} \leq\tpair{t}{f}$  in $\tors{A}$. They determine an  interval in $\tors{A}$ that we denote  by $[\mathbf{u}, \mathbf{t}]$.
Again, we adopt  the notation of \S~\ref{hearts}, together with the analogous notation $\mathfrak{u}, \sigma_{\mathfrak u}, \N_{\mathfrak u}$ etc. for the torsion pair $\tpair{u}{v}$. Recall that 
$\mathbb T_{\mathfrak t}=\mathbb{D}_{\mathfrak{t}^-}$ and $  \mathbb T_{\mathfrak u}=
 \mathbb{D}_{\mathfrak{u}^-}$ are HRS-tilts of the standard t-structure at $\mathfrak t$ and $\mathfrak u$, respectively. In fact, also $\mathbb T_{\mathfrak t}$ and $  \mathbb T_{\mathfrak u}$ are related by an HRS-tilt, as stated below and illustrated by the commutative diagram
$$\xymatrix{
  \mathbb{D}\ar@{~>}[d]_{\mathfrak{u}^-} \ar@{~>}[r]^{\mathfrak{t}^-} &  \; \mathbb T_{\mathfrak t}  \\
  \mathbb T_{\mathfrak u} \ar@{~>}[ur]_{\mathfrak{s}^-}
}$$
 
\begin{proposition} \cite[Prop.~7.5]{ALSV}  
The t-structure $\mathbb T_{\mathfrak t}$ in $\Der{A}$ is a right HRS-tilt of  the t-structure $\mathbb T_{\mathfrak u}$ at the torsion pair $\mathfrak s=(\mathcal S, \R)$ in the heart $\heart{\mathfrak{u}}$  with torsion class $\mathcal S = \T\cap \V$.
\end{proposition}
When $\N_\mathfrak{t}$ and $\N_\mathfrak{u}$ are related by mutation, and $\E = {\N_\mathfrak{t}}\cap {\N_\mathfrak{u}}$, one can use the exchange triangle (\ref{eq: exchange triangle}) to verify that $\mathbb T_{\mathfrak t}$ is a right HRS-tilt of $  \mathbb T_{\mathfrak u}$ at the torsion pair $({}^{\perp_0}\H{\mathfrak{u}}(\E), \Cogen{\H{\mathfrak{u}}(\E)})$  in  $\heart{\mathfrak{u}}$. Now recall that a torsion pair in a Grothendieck category is  hereditary  if and only if it is cogenerated by a class of injective objects, and $\H{\mathfrak u}(\E)$ is a class of injective objects, even a closed set in $\Sp{\heart{\mathfrak u}}$, by Proposition~\ref{propheart}. We obtain the following characterization of mutation. 
\begin{proposition}\label{heredtp} \cite[Thm.~3.5]{ALSV}  
The maximal rigid sets $\N_\mathfrak{t}$ and $\N_\mathfrak{u}$ are related by mutation if and only if the torsion pair $\mathfrak s=(\mathcal S, \R)$ in $\heart{\mathfrak{u}}$ is hereditary.
\end{proposition}

In the situation above $\mathfrak s$
is a hereditary torsion pair of finite type, so by general results on localisation theory in locally coherent Grothendieck categories \cite{Herzog:97, Krause:97}, it is generated by a Serre subcategory  of $\fp{\heart{\mathfrak u}}$, i.~e.~a subcategory that is closed under subobjects, quotients, and extensions.  Proposition~\ref{heredtp} can then be rephrased by saying that $\N_\mathfrak{t}$ and $\N_\mathfrak{u}$ are related by mutation if and only if $\mathcal S\cap\fp{\heart{\mathfrak u}}=\mathbf{t}\cap\mathbf{v}$ is a Serre subcategory of $\fp{\heart{\mathfrak{u}}}$. 
Observe further that $\mathcal S\cap\fp{\heart{\mathfrak u}}$, being  also contained in the length category $\mod{A}$,  is a Serre subcategory if and only if it can be realised as the class $\filt\B$ of objects  that admit a finite filtration by a set $\B$ of finitely presented simple objects in $\heart{\mathfrak u}$. In other terms, $\B$ is the semibrick formed by the modules that are torsionfree almost torsion with respect to $\mathfrak u$ and belong to the torsion class $\mathbf t$. This means precisely that $\mathbf{t}\cap\mathbf{v}$ is a wide subcategory of $\mod{A}$, see \cite[Thm.~1.4]{AP} or \cite[Thm.~8.8]{ALSV}. An interval
 $[\mathbf u,\mathbf t]$  in $\tors{A}$ with this property is said to be a \textbf{wide  interval}. 

Finally, let us point out that the modules that are torsionfree almost torsion with respect to $\mathfrak u$ and belong to the torsion class $\mathbf t$ coincide with  the modules that are torsion almost torsionfree with respect to $\mathfrak t$ and belong to the torsionfree class $\mathbf v$, and also with the simple objects in the wide subcategory $\mathbf{t}\cap\mathbf{v}$, see \cite[Lem.~9.6]{ALSV}. Mutation is thus determined by a semibrick $\B$ which gives rise to two sets of finitely presented simple objects in the associated hearts, the set $\B$ in $\heart{\mathfrak u}$, and the set $\B[-1]$ in  $\heart{\mathfrak t}$. From this perspective, the HRS-tilt   at the hereditary torsion pair $\mathfrak{s}$ in $\heart{\mathfrak u}$ cogenerated by $\H{\mathfrak{u}}(\E)$  amounts to exchanging the injective envelopes of these simples. More precisely, the $\N_\mathfrak{u}$-injective envelopes of the simples in $\B$ are replaced by the
$\N_\mathfrak{t}$-injective envelopes of the simples in $\B[-1]$. In other words, we exchange critical points of $\N_\mathfrak{u}$
with special points  of $\N_\mathfrak{t}$. 
 This is summarised by the following result.
 
\begin{theorem}\cite[Thm.~8.8, Thm.~9.7, Lem.~9.11]{ALSV} \label{mut} The following statements are equivalent.
\begin{enumerate}
\item The maximal rigid sets $\N_\mathfrak{t}$ and $\N_\mathfrak{u}$ are related by mutation
\item $[\mathbf{u}, \mathbf{t}]$ is a wide interval in $\tors{A}$.

\item There exists a semibrick $\B_\mathfrak{u}$  formed by torsionfree almost torsion modules with respect to $\mathfrak{u}$ such that $\mathbf{t} \cap \mathbf{v}=\filt{\B_\mathfrak{u}}$.
\item There exists a semibrick $\B_\mathfrak{t}$ formed by torsion almost torsionfree modules with respect to $\mathfrak{t}$ \mbox{such that $(\mathbf{t}\cap \mathbf{v})[-1] =\filt{\B_\mathfrak{t}[-1]}$.}
\end{enumerate}
In this case $\B := \B_\mathfrak{u} = \B_\mathfrak{t}$, and the bijection $\tt m\colon \N_\mathfrak{u}\to \N_\mathfrak{t}$  witnessing the mutation takes the \mbox{$\N_\mathfrak{u}$-injective envelope} of $B$ to the $\N_\mathfrak{t}$-injective envelope of $B[-1]$ for each $B\in\B$  and fixes the remaining elements of $\N_\mathfrak{u}$.
\end{theorem}

Let us now turn to the Hasse quiver of $\tors{A}$.
There is an  arrow  $\mathbf t\to\mathbf u$ in the Hasse quiver if and only if the inclusion $\mathbf u\subset\mathbf t$ is  minimal. We will then say that $ \tpair{t}{f}$ \textbf{covers} $\tpair{u}{v}$  in $\tors{A}$.
This is the case when the wide subcategory $\mathbf{t} \cap \mathbf{v}$ has exactly one simple object, which is precisely  the brick label of the arrow $\mathbf t\to\mathbf u$ considered in  \cite{BarnardCarrollZhu:19, DemonetIyamaReadingReitenThomas:23}. 

\begin{corollary} \cite[Cor.~9.14]{ALSV}\label{cover}
The following statements are equivalent.
\begin{enumerate}
\item $ \tpair{t}{f}$ covers $\tpair{u}{v}$  in $\tors{A}$.
\item
 The maximal rigid sets $\N_\mathfrak{t}$ and $\N_\mathfrak{u}$ are related by irreducible mutation.
 \item  $\mathbf{t} \cap \mathbf{v} = \filt{B}$ where $B\in\mod{A}$ is torsionfree almost torsion with respect to $\mathfrak{u}$.
 \item $\mathbf{t} \cap \mathbf{v} = \filt{B}$ where $B\in\mod{A}$ is torsion almost torsionfree  with respect to $\mathfrak{t}$.
\end{enumerate}

Moreover, under the equivalent conditions above, if $\N_\mathfrak{u}\setminus\N_\mathfrak{t} = \{\lambda\}$ 
and $\N_\mathfrak{t}\setminus\N_\mathfrak{u} = \{\rho\}$, then 
 $\lambda$ is the $\N_\mathfrak{u}$-injective envelope of the simple object $B$ in $\heart{\mathfrak{u}}$  and $\rho$ is the $\N_\mathfrak{t}$-injective envelope of the simple object $B[-1]$ in $\heart{\mathfrak{t}}$.
\end{corollary}
\begin{proof}
Condition (1) means that the  Serre subcategory $\mathcal S\cap\fp{\heart{\mathfrak u}}$ of $\fp{\heart{\mathfrak{u}}}$ is given by a single simple object $B$. The hereditary torsion pair $(\mathcal S,\R)$ in ${\heart{\mathfrak{u}}}$ is then generated by $B$, and all but one indecomposable injective objects are in the  torsionfree class  $\R=B^{\perp_0}$, namely all but the injective envelope of $B$. Now recall that $\Sp{\heart{\mathfrak{u}}}$ is in bijection with ${\N_\mathfrak{u}}$ via the cohomological functor $\H{\mathfrak{u}}$, and moreover,
$(\mathcal S,\R)$ is cogenerated by the closed set ${\H{\mathfrak{u}}(\E)}$ in $\Sp{\heart{\mathfrak{u}}}$, where $\E={\N_\mathfrak{t}}\cap {\N_\mathfrak{u}}$. We can then conclude that ${\N_\mathfrak{u}}\setminus {\N_\mathfrak{t}}={\N_\mathfrak{u}}\setminus\E$ contains precisely one object, namely the ${\N_\mathfrak{u}}$-injective envelope of $B$.
This proves (1)$\Rightarrow$(2), and the reverse implication is shown similarly. The remaining statements are immediate consequences of Theorem~\ref{mut}.
\end{proof}
\begin{example} \label{kron2}
Let $A$ be the Kronecker algebra, and let $\tpair{t}{f}=\tpair{\add{\mathbf{r}\cup\mathbf{q}}}{\add{\mathbf p}}$ and  $\tpair{u}{v}=\tpair{\add{\mathbf{q}}}{\add{\mathbf p\cup\mathbf{r}}}$ be the  two torsion pairs from  Example~\ref{kron}. They form a wide interval, as $\mathbf{t}\cap\mathbf{v}$ equals the wide subcategory $\add{\mathbf{r}}$ formed by the regular modules. 

Let us determine the maximal rigid sets associated to the two torsion pairs.
Since $A$ is hereditary,   $\Zg{\Der{A}}$ is determined, up to shift, by the Ziegler spectrum $\Zg{A}$ of $\Mod{A}$, which consists of \begin{itemize} 
\item the indecomposable finite dimensional modules, \item
a Pr\"ufer module $S[\infty]$ and
 an adic module $S[-\infty]$ for each simple regular module $S$, 
 \item 
 the generic module $G$.
 \end{itemize}
For details we refer e.g.~to \cite{CB5}.

The  torsion pair in $\Tors{A}$ associated to $\tpair{u}{v}$ is the torsion pair $\mathfrak u=\Tpair{U}{V}$  cogenerated by the collection of all Pr\"ufer modules,  and the associated maximal rigid set $\N_\mathfrak{u}$ is given by $G$ and all Pr\"ufer modules. 
Every torsionfree almost torsion module with respect to $\mathfrak u$, that is, every simple regular module $S$
admits a  $\V^{\perp_1}$-approximation sequence
$
\begin{tikzcd}
0 \arrow[r] & S \arrow[r, "f"] & S[\infty] \arrow[r, ""] & S[\infty] \arrow[r] & 0
\end{tikzcd}
$. This shows that the Pr\"ufer modules are the critical points of $\N_\mathfrak{u}$.

The  torsion pair in $\Tors{A}$ associated to $\tpair{t}{f}$ is the torsion pair $\mathfrak{t}=\Tpair{T}{F}$ cogenerated by the generic module $G$,  and the associated maximal rigid set $\N_\mathfrak{t}$ is given by $G$ and all adic modules. 
Every torsion almost torsionfree  module with respect to $\mathfrak t$, that is, every simple regular module $S$
admits an $\F$-approximation sequence
$
\begin{tikzcd}
0 \arrow[r] & S[-\infty] \arrow[r, ""] & S[-\infty] \arrow[r, "g"] & S \arrow[r] & 0
\end{tikzcd}
$. This shows that
 the adic modules are the special points of $\N_\mathfrak{t}$.

Now the mutation corresponding to the wide interval $[\mathbf u,\mathbf t]$ consists in replacing the Pr\"ufer modules by the adic modules.\end{example}

\section{Wide intervals}

To every wide interval $[\mathbf{u}, \mathbf{t}]$ in $\tors{A}$ we can associate the  set $\E={\N_\mathfrak{t}}\cap {\N_\mathfrak{u}}$ which is closed and  rigid in $\ZgInt{A}$. It consists of the points that are not touched by the mutation process, as  seen above. The aim of this section is to show that this assignment defines a bijection between wide intervals and  closed rigid sets in $\ZgInt{A}$.

Note that every closed rigid set $\M$ in $\ZgInt{A}$ induces two torsion pairs of finite type in $\Mod{A}$.  
The first torsion pair $\mathfrak{u}_\M=\Tpair{U}{V}$  has torsionfree class $\V=\bigcap_{\mu\in\M}\C_\mu$, and the second is the  torsion pair $\mathfrak{t}_\M=\Tpair{\T}{\F}$ cogenerated by the set ${\H{}(\M)}$. Since $\M$ is rigid, one can use  Lemma~\ref{csigma} to show that $\F=\Cogen{\H{}(\M)}\subseteq\V$, so $\U\subseteq\T$. 

Now we can again consider the HRS-tilts of the standard t-structure at $\mathfrak{u}_\M$ and $\mathfrak{t}_\M$, respectively, and remember that they are related by a HRS-tilt 
at the torsion pair $\mathfrak s=(\mathcal S, \R)$ in the heart of $\mathbb D_{\mathfrak u_\M^-}$ with torsion class $\mathcal S = \T\cap \V$. 
It turns out that  $\mathfrak{s}$ is a hereditary torsion pair of finite type,
so we conclude from Proposition~\ref{heredtp} and Theorem~\ref{mut} that $\mathfrak{u}_\M$ and $\mathfrak{t}_\M$ restrict to a wide interval which we will denote by $[\mathbf u_\M,\mathbf t_\M]$.
One verifies that, under the lattice isomorphism   between $\tors{A}$ and $\CosiltZg{A}$,  this interval is  taken to the set of completions of $\M$ to a maximal rigid set.

\begin{theorem}\cite[Thm.~5.1,Cor.~5.10]{ALS2}\label{wide} The assignment $[\mathbf{u}, \mathbf{t}]\mapsto\N_\mathfrak{t}\cap\N_\mathfrak{u}$ defines a bijection between the wide intervals in $\tors{A}$ and the closed rigid sets in $\ZgInt{A}$. The inverse map assigns to a closed rigid set $\M$  the wide  interval  $[\mathbf{u}_\M,\mathbf{t}_\M]$ in $\tors{A}$, which is formed by all torsion classes whose associated maximal rigid set  contains  $\M$.
\end{theorem}

Let $\ClRigid{A}$ be the collection  of all closed rigid sets in $\ZgInt{A}$. We say that $\M$ in $\ClRigid{A}$ is \textbf{almost complete} if $\M$ is not maximal, but every  set in $\ClRigid{A}$ properly containing $\M$ is maximal. This generalises the notion of an  almost complete silting subcategory from \cite{AiharaIyama}.

\begin{corollary}\cite[Thm.~6.2]{ALS2}. The following statements are equivalent for  $\M$ in $\ClRigid{A}$.
\begin{enumerate}
\item $\M$ is almost complete.
\item $\M$ can be completed in exactly two ways to a maximal rigid set $\N$.
\item There is an arrow  $\mathbf{t}_\M\to \mathbf{u}_\M$ in the Hasse quiver of $\tors{A}$. 
\end{enumerate}
\end{corollary}

We have seen in Section~\ref{muta} that mutation consists in exchanging  $\N$-injective envelopes of finitely presented simple objects in hearts. 
Mutation  can thus only change neg-isolated points of $\N$, and in fact, only those  that are associated  to finitely presented simple objects. The next result aims at an intrinsic description of these points, using the Ziegler topology.
We will say that a point $\mu$ in a maximal rigid set $\N$ is \emph{mutable} if there exists a maximal rigid set 
$\N'$ that arises from $\N$ by removing $\mu$ and replacing it with a point $\nu$    related by mutation with $\mu$. 

\begin{corollary}\cite[Cor.~6.3]{ALS2} Let $\N$ be a maximal rigid set. A point  $\mu\in\N$ is mutable if and only if the set there exists a unique  maximal rigid set $\N'\not=\N$ containing  $\N \setminus \{\mu\}$. 
In this case, $\N \setminus \{\mu\}=\N\cap\N'$, and $\mu$ is an isolated point of $\N$ in the Ziegler subspace topology on $\N$.\end{corollary}

\begin{remark}\cite[\S 6.1 and 6.2]{ALS2} {\rm For many finite dimensional algebras also the converse statement holds true, that is,  the mutable points in $\N$ are precisely the isolated ones. This holds for example when $A$ has   finite representation type, or is a tame hereditary algebra, or a domestic string algebra. More cases are listed in \cite[Rem.~6.17]{ALS2}.}\end{remark}

\begin{example} \label{kron3}
In Example~\ref{kron2}, the closed rigid set associated to the wide interval $[\mathbf u,\mathbf t]$ is the set $\{G\}$, and  $G$ is contained in all maximal rigid sets associated to  torsion pairs $\tpair{t'}{f'}$ with $\mathbf u\subseteq\mathbf{t'}\subseteq\mathbf t$.
We will see in Section~\ref{semistable} that this  is an instance of a much more general phenomenon.
\end{example}
 
\section{Applications} 
Throughout this section $A$ denotes a finite dimensional algebra. 
\subsection{Bricks and grains.}\label{grains}
It was shown in \cite{DemonetIyamaJasso:19} that there is a bijection between left finite bricks and  $\tau$-rigid modules, or dually, right finite bricks and $\tau^{-}$-rigid modules. Here a brick $B$ is said to be \textbf{right finite} if   the torsion pair
$\tpair{u_{\rm B}}{v_{\rm B}}=({}^{\perp_0}B\cap\mod{A},\tilde{F}(B))$    cogenerated by $B$ in $\mod{A}$ is functorially finite. 

In order to relax this finiteness assumption, we need to generalise the notion of a $\tau^{-}$-rigid module. We say that  an indecomposable pure-injective module $N$ is a \textbf{grain} if its  minimal injective copresentation $\mu_N$ is rigid, i.e. $\{\mu_N\}$ is a rigid set in $\Der{A}$, or equivalently, $N$ is the zero cohomology of a point in some maximal rigid set. Let us establish a relation between grains and bricks.
 
 Given a brick $B$, we know from \cite{BarnardCarrollZhu:19, DemonetIyamaReadingReitenThomas:23} that 
  the torsion class $\mathbf{u_{\rm B}}$  in the torsion pair $\tpair{u_{\rm B}}{v_{\rm B}}$  cogenerated by  $B$ is \textbf{completely meet irreducible}, i.e.~it cannot be written as an intersection of a family of strictly bigger torsion classes. Moreover,  every completely meet irreducible torsion class is of this form.  Notice that $B$ is the unique finite dimensional brick  which is  torsionfree almost torsion with respect to the corresponding torsion pair $\Tpair{U_{\rm B}}{V_{\rm B}}$ in $\Tors{A}$, and by Proposition~\ref{las} there is an  associated critical module $N$ in $\V_B$. By definition, $N$ is a grain. Furthermore, one checks that the torsion pair $(\t_N,\f_N)$ in $\mod{A}$ with torsion  class $\t_N:=\{X\in\mod{A}\mid \Hom{A}(X,N)=0\}$ agrees with $\tpair{u_{\rm B}}{v_{\rm B}}$, so $\t_N$ is completely meet irreducible.

We thus obtain a map from $\brick{A}$  to the class $\Ri(A)$  formed by the  \srigid s $N$ such that the  torsion class $\t_N$ is completely meet irreducible. It defines a bijection with inverse map taking a \srigid\  $N$ to the unique brick $B_N$ which cogenerates the torsion pair $(\t_N,\f_N)$.

 Observe that a finite dimensional module $N$ lies in $\Ri(A)$ if and only if it is $\tau^{-}$-rigid.
Moreover, in this case the brick $B_N$ is isomorphic to the socle of $N$ over its endomorphism ring $\End{A}(N)$. The assignment $N\mapsto B_N$ is thus dual to the map considered in \cite[Thm.~4.1]{DemonetIyamaJasso:19}.

\begin{theorem}\cite[Thm.~B]{ALS1}\label{B}
There is a bijection 
	$\brick{A}\longrightarrow \Ri(A)$ which restricts to a bijection between right finite bricks and $\tau^{-}$-rigid modules. 
\end{theorem}

\subsection{Brick-finite algebras.}\label{brickfin}
Recall that the algebra $A$ is \textbf{brick-finite} if $\brick{A}$ is a finite set. It was proven by Demonet, Iyama and Jasso in \cite{DemonetIyamaJasso:19} that $A$ is brick-finite if and only if all torsion classes are functorially finite.   Sentieri then showed in \cite{Sentieri:22} that any  brick-infinite algebra admits a torsion pair $\tpair{u}{v}$ with a maximal non functorially finite torsion class $\mathbf{u}$, and such torsion pair is  \textbf{locally maximal}, i.e. there are no
torsion pairs $\tpair{t}{f}$ in $\tors{A}$ covering  $\tpair{u}{v}$. By Corollary~\ref{cover} this means that there is no $B \in\mod{A}$ which is torsionfree almost torsion with respect to $\tpair{u}{v}$, or in other words, there is no finitely presented simple torsion object in the associated heart. On the other hand, we have seen in \S~\ref{hearts} that the heart always admits simple torsion objects. Thus there must be a torsionfree almost torsion  module $B$ which is  infinite dimensional. This is analogous to Auslander's classical result proving the existence of large indecomposable modules over representation-infinite algebras.
\begin{theorem}\cite[Thm.~3.14]{Sentieri:22}\cite[Cor.~C]{ALS1} The following statements are equivalent.
\begin{enumerate}
\item  $A$ is brick finite.
\item All bricks in $\Mod{A}$ are finite dimensional.
\item The torsion class $\t_N$ is completely meet irreducible in $\tors{A}$ for every grain $N$.
\end{enumerate}
\end{theorem}

\subsection{Semistable torsion pairs.} \label{semistable}

Let  $\mathrm{inj}(A)$ denote the category of finitely generated injective $A$-modules and ${\rm K}_0(\mathrm{inj}(A))$ its Grothendieck group. We fix a g-vector  $\Theta\in{\rm K}_0(\mathrm{inj}(A))$ and regard it as a linear form  on the real Grothendieck group ${\rm K}_0(\mod{A})\otimes\mathbb R$. Consider the  torsion pairs $(\overline{\t_\Theta}, \f_\Theta)$ and $(\t_\Theta, \overline{\f_\Theta})$ in $\tors{A}$ given by

$\overline{\t_\Theta}:=\{M\in\mod{A}\mid \Theta(M')\le 0 \text{ for all nonzero factor modules } M' \text{ of } M\}$,

$\overline{\f_\Theta}:=\{M\in\mod{A}\mid \Theta(M')\ge 0 \text{ for all nonzero submodules } M' \text{ of } M\}$,

and $\t_\Theta$ and $\f_\Theta$ defined correspondingly by strict inequalities.
Notice that $[\t_\Theta,\overline{\t_\Theta}]$ is a wide interval, as $\overline{\t_\Theta}\cap\overline{\f_\Theta}$ is the category of $\Theta$-semistable modules in the sense of \cite{King}. We want to determine the associated closed rigid set $\M_\Theta$ under the bijection in Theorem~\ref{wide}.

 We can write $\Theta=[I_0]-[I_{1}]$, where $I_0,I_{1}\in\mathrm{inj}(A)$ have no common non-zero summands and are unique up to isomorphism, and consider the irreducible  variety 
$\Hom{}(\Theta):=\Hom{A}(I_0,I_1)$. 
By \cite{DF} we can assume without loss of generality that $\Theta$ is generically indecomposable, i.e.~all elements in a suitable dense open set of $\Hom{}(\Theta)$ are indecomposable. Moreover, if we assume that $A$ is of tame representation type, then it follows from \cite{PY,GLS} that there are two cases: either
(i) $\Theta$ is rigid,  or
(ii) there is a dense open subset  $U$ of $\Hom{}(\Theta)$ such that the kernels of the morphisms in $U$ form a one-parameter family of bricks.
We refer to \cite{DF,PY,GLS} for details.

\begin{example} \label{kron4} In Examples~\ref{kron2} and~\ref{kron3} we have $[\mathbf{u},\mathbf{t}]=[\t_\Theta,\overline{\t_\Theta}]$ for $\Theta=(1,-1)^T$, and the one-parameter family of bricks mentioned above is given by the simple regular modules.\end{example}

In case (i), we know from \cite{AsaiIyama} that there is a cosilting complex $\mu$ with finite dimensional injective terms  such that $\M_\Theta=\{\mu\}$. 
In case (ii),  we obtain a family of tubes $\Pi$ in the Auslander-Reiten quiver of $A$, and by  \cite{Krause} the closure of $\Pi$ in the Ziegler topology contains a generic module $G$. Combining the methods outlined  in these notes with the main  result of \cite{Bautista} we obtain 
 
\begin{theorem}\cite{ALP} Let $A$ be a tame algebra and let $\Theta$ be generically indecomposable and not rigid.  Then  $\M_\Theta=\{\mu_G\}$, where $\mu_G$ is the minimal injective copresentation of $G$. Moreover, $\overline{\t_\Theta}$ is  locally maximal  and $\t_\Theta$ is  locally minimal in $\tors{A}$. \end{theorem}

\bibliographystyle{plain}
\bibliography{refs}

\end{document}